\documentclass{amsart}
\usepackage{amsmath}
\usepackage{amssymb}
\usepackage{amsfonts}
\usepackage{amsthm}
\usepackage{enumerate}
\usepackage[all]{xy}

\newcommand{\mb}{\mathbb}

\newcommand{\ol}{\overline}
\newcommand{\ra}{\rightarrow}

\newcommand{\ox}{\ol{x}}

\newcommand{\iso}{\approx}

\newcommand{\Gal}{\text{Gal}}

\newcommand{\Z}{\mb{Z}}

\newcommand{\Cl}{\text{Cl}}
\newcommand{\F}{\mb{F}}

\newcommand{\Q}{\mb{Q}}

\newcommand{\<}{\langle}
\renewcommand{\>}{\rangle}

\newcommand{\vare}{\varepsilon}

\renewcommand{\l}{\ell}

\begin{document}

\newtheorem{Theo}{Theorem}
\newtheorem{PTheo}{Pre-Theorem}
\newtheorem{Prop}[Theo]{Proposition}
\newtheorem{Exam}[Theo]{Example}
\newtheorem{Lemma}[Theo]{Lemma}

\newtheorem{Claim}{Claim}
\newtheorem{Cor}[Theo]{Corollary}
\newtheorem{Conj}[Theo]{Conjecture}
\newtheorem*{ToUnderstand}{To Understand}
\newtheorem*{ToProve}{To Prove}
\theoremstyle{definition}
\newtheorem{Remark}[Theo]{Remark}

\theoremstyle{definition}
\newtheorem{Defn}[Theo]{Definition}
\title{A Golod-Shafarevich Equality and $p$-Tower Groups}
\author{Cam McLeman}
\thanks{Pre-reviewed version of \emph{A Golod-Shafarevich equality and $p$-tower groups.}  J. Number Theory  129  (2009),  no. 11, 2808--2819.  MR2549535 (2010i:11166)}
\address{Willamette University, 900 State Street, Collins Science Center.  Salem, OR.  97304}
\email{cmcleman@willamette.edu}
\date{January 17, 2009} 

\begin{abstract}
All current techniques for showing that a number field has an infinite
$p$-class field tower depend on one of various forms of the
Golod-Shafarevich inequality.  Such techniques can also be used to
restrict the types of $p$-groups which can occur as Galois groups of
finite $p$-class field towers.  In the case that the base field is a
quadratic imaginary number field, the theory culminates in showing
that a finite such group must be of one of three possible presentation
types.  By keeping track of the error terms arising in standard proofs
of Golod-Shafarevich type inequalities, we prove a Golod-Shafarevich
equality for analytic pro-$p$-groups.  As an application, we further
work of Skopin (\cite{Sk73}), showing that groups of the third of the
three types mentioned above are necessarily tremendously large.
\end{abstract}
\maketitle 

\section{Introduction}

All current techniques for showing that a number field has an infinite
$p$-class field tower depend on one of various forms of the
Golod-Shafarevich inequality, a purely group-theoretic result relating
(among other invariants) the generator rank $d$ and relation rank $r$
of an analytic pro-$p$-group.  Even relatively weak forms of the
theorem (e.g., the famous inequality $r>\frac{d^2}{4}$) provide the
first examples of fields with infinite $p$-class field towers.  A much
stronger form of the inequality due to Koch (see Remark
\ref{Kochform}) relates finer invariants describing a group's relation
structure.

At the heart of the proofs of these stronger forms is the Fox
differential calculus, which gives rise to a sequence of inequalities
relating various invariants attached to a pro-$p$-group.  Our first
contribution is to introduce and analyze a new set of obstruction
invariants, measuring the extent to which these inequalities fail to
be equalities.  In section 3, we carry out a standard proof of the
Golod-Shafarevich inequality, now with these obstruction invariants in
place.  In conjunction with a theorem of Jennings on dimension factors
of $p$-groups, this gives our principal result, a new
Golod-Shafarevich \emph{equality} (Theorem \ref{GSEQ}).  One immediate
corollary (Corollary \ref{strict}), stemming from lower bounds placed
on the obstruction invariants, is a strict improvement of the stronger
form of the Golod-Shafarevich inequality mentioned above.  A principal
benefit of Theorem \ref{GSEQ} over similar results is that one can
extract information about the \emph{order} of the group in question.
We take advantage of this in Section 4, where we apply the theorem to
the class field tower group over a quadratic imaginary number field.
In this case, the relation structure for non-trivial class field tower
groups is of one of three types (see Theorem \ref{KV}), the last of
which is not known to occur for \emph{any} finite $p$-group.  An
application of Theorem \ref{GSEQ} puts a rather large lower bound on
the order of such a group.

\section{Background}

Let $K$ be a number field, and $p$ a prime number.  Denote by
$K^{(1)}$ the Hilbert $p$-class field of $K$, i.e., the maximal
abelian $p$-extension of $K$ which is unramified at all primes.  Class
field theory tells us that this is a finite Galois extension of $K$
whose Galois group is isomorphic to the $p$-primary part of the ideal
class group of $K$.  Iterating this procedure constructs the $p$-class
field tower over $K$:
\begin{align*}
K=K^{(0)}\subset K^{(1)}\subset K^{(2)}\subset K^{(3)}\subset \cdots,
\end{align*}
where for $i\geq 0$, $K^{(i+1)}$ is the Hilbert $p$-class field of
$K^{(i)}$.  Let $K^{(\infty)}$ denote the union of the fields in the
tower.  An important question, open for most number fields, is whether
or not $K$ admits a finite extension with class number prime to $p$ --
a subtle arithmetic condition which arose, for example, in Kummer's
work on the first case of Fermat's Last Theorem for regular primes.
This embeddability condition is equivalent to the question of whether
or not the $p$-class field tower over $K$ stabilizes, i.e., whether or
not there exists a positive integer $\l:=\l_p(K)$ such that
$K^{(\l)}=K^{(\l+1)}=K^{(\l+2)}=\cdots$.  We call the smallest such
$\l$ the \emph{$p$-tower length} of $K$, and set $\l=\infty$ if no
such integer exists.  Defining the \emph{$p$-tower group over $K$} by
$G_K^\infty:=\Gal(K^{(\infty)}/K)$, the observation that each
extension $K^{(i+1)}/K^{(i)}$ is finite implies that $\l_p(K)$ is
finite if and only if $G_K^\infty$ is.  We will thus turn our
attention to studying the pro-$p$-groups $G_K^\infty$, some of the
``most mysterious objects in algebraic number theory.'' (\cite{Wi93})

The study of such groups remains in the slightly paradoxical
situation that while even the finiteness of $G_K^\infty$ for a given
$K$ is difficult to decide, we have rather detailed information on
other aspects of its structure.  Namely, work of Shafarevich
(\cite{Sh63}) calculates the generator and relation ranks
\begin{align*}
d:=\dim_{\F_p}H^1(G_K^\infty,\F_p)\quad\quad\text{ and }\quad\quad r:=\dim_{\F_p}H^2(G_K^\infty,\F_p)
\end{align*}
in terms of arithmetic information of $K$, and work of Koch
(\cite{Ko69}), Venkov (\cite{KV74}), and more recently Vogel
(\cite{Vo04}), give information on the specific form of those
relations.  The standard, and essentially only, way of demonstrating a
$p$-class field tower to be infinite is by combining these
calculations with (one of various forms of) the Golod-Shafarevich
inequality, which we turn to next.

\begin{Defn}
Let $G$ be a pro-$p$-group and let $\F_p[G]$ be its group ring over
$\F_p$.  The augmentation map $\vare:\F_p[G]\ra \F_p$, given by
$\vare(\sum a_ig_i)=\sum a_i$, is a surjective homomorphism whose
kernel $I$ is called the \emph{augmentation ideal} of $\F_p[G]$ (or
just of $G$).  Define the \emph{$n$-th modular dimension subgroup
$G_n$ of $G$} by
\begin{align*}
G_n=\{g\in G\mid g-1\in I^n(G)\}.
\end{align*}
The filtration $G=G_1\supset G_2\supset\cdots$ of $G$ by its modular
dimension subgroups is called the \emph{Zassenhaus filtration} of $G$.
One checks easily (e.g., \cite{Ko02}, Theorem 7.12) that if $g,h\in
G_n$, then $[g,h]\in G_{n+1}$ and $g^p\in G_{np}\subset G_{n+1}$, so
that the quotients $G_n/G_{n+1}$ are $\F_p$-vector spaces for all $n$.
We define the (modular) \emph{dimension factors} of $G$ by
$a_n(G):=\dim_{\F_p}G_n/G_{n+1}$.
\end{Defn}

If $G$ is a $d$-generated, $r$-related
pro-$p$-group, we call a presentation
\begin{align*}
\xymatrix{1\ar[r]&R\ar[r]&F\ar[r]&G\ar[r]&1}
\end{align*}
\emph{minimal} if $F$ is a free pro-$p$-group on $d$ generators and
$R$ is generated as a normal subgroup of $F$ by $r$ elements.  The
most commonly cited form of the Golod-Shafarevich inequality places
the lower bound $r>\frac{d^2}{4}$ on the number of relations $r$
required to force a $d$-generated pro-$p$-group finite.  A more
refined version observes that relations lying deeper in the Zassenhaus
filtration contribute less to keeping a group finite, and hence more
such relations would be required.

\begin{Theo}[\cite{Ko69}]\label{medgs}
Suppose $1\ra R\ra F\ra G\ra 1$ is a minimal presentation for a
pro-$p$-group $G$, and that $R\subset F_m$.  If $G$ is finite, then
\begin{align*}
r>\frac{d^m(m-1)^{m-1}}{m^m}.
\end{align*}
\end{Theo}
\begin{Remark}
The bound $r>\frac{d^2}{4}$ now follows from the observation that one
has $r_1=0$ for a minimal presentation, since a relation in level one
would kill off a generator, contradicting that $G$ was minimally
$d$-generated.
\end{Remark}

The results referenced above combine with the Golod-Shafarevich
equality to give a particularly strong answer in the case that the
base field $K$ is a quadratic imaginary number field.  Let
$d_p\Cl(K):=\dim_{\F_p}\Cl(K)/p$ be the $p$-rank of the class group of
$K$.  The calculation of Shafarevich (\cite{Sh63}, Theorem 1) gives
that $r(G_K^\infty)=d(G_K^\infty)=d_p\Cl(K)$, and a result of Koch and
Venkov (\cite{KV74}, Theorem 2) uses the fact that $G_K^\infty$ is a
so-called Schur-$\sigma$ group to conclude that $R\subset F_3$.  The
Golod-Shafarevich inequality gives in this case that
\begin{align*}
d>\frac{4d^3}{27}.
\end{align*}
Since this inequality is violated for $d\geq 3$, we find that $K$ has
an infinite $p$-class field tower whenever the $p$-rank of $\Cl(K)$ is
at least 3.  Further, it is easy to show that if the $p$-rank of
$\Cl(K)$ is less than or equal to one, then $K$ has a finite class
field tower.  The only remaining case, where $d=r=2$, will be
discussed in Section \ref{quadim}, after we prove a stronger form of
the Golod-Shafarevich result.

\section{A Golod-Shafarevich Equality}

Our contribution to the theory will be to introduce and analyze a
series of invariants (dubbed $e_n$ below) that one can attach to a
finitely-generated pro-$p$-group to find the source of the
``inequality'' in the Golod-Shafarevich inequality.  These invariants,
which admit an interpretation in terms of a non-commutative Jacobian
map on formal power series, can be shown to supply a non-trivial error
term, leading to a refinement of the inequality.  As the beginning of
the proof will closely follow that given by Koch in the appendix of
\cite{Ha78}, we will largely omit details until the two proofs differ.

Let $G$ be a $d$-generated pro-$p$-group, and consider a minimal
presentation
\begin{align*}
\xymatrix{ 1\ar[r]&R\ar[r]^\iota&F \ar[r]^\phi& G\ar[r]&1 }
\end{align*}
of $G$ as a pro-$p$-group.  We choose lifts $\{\sigma_i\}_{i=1}^d$ to
$F$ of a minimal generating set for $G$, and go through an inductive
procedure to choose a generating system of relations which is minimal
with respect to the Zassenhaus filtration.  Namely, define
$R_0=\emptyset$, and for $n\geq 1$, let
$R_n=R_{n-1}\cup\{\rho_{n,1},\ldots,\rho_{n,{r_n}}\}$ where the
relations $\rho_{n,1},\rho_{n,2}\ldots,\rho_{n,{r_n}}$ are chosen so
that they and $R_{n-1}$ constitute a minimal system of generators for
$RF_{n+1}/F_{n+1}$.  In the process, we have also defined invariants
$r_k$ representing the number of relations of level $k$ in a minimal
presentation for $G$.  Note that $\sum_{k=1}^\infty r_k=r$.  The
(completed) group ring $\F_p[F]$ is isomorphic to the ring
$\F_p(d):=\F_p\{\{x_1,\ldots,x_d\}\}$ of formal power series in $d$
non-commuting variables over $\F_p$, the isomorphism being the linear
extension of the map sending $\sigma_i$ to $1+x_i$.

The map $\phi:F\ra G$ above extends naturally to a surjection (which
we also call $\phi$)
\begin{align*}
\xymatrix{ \F_p[F]\ar[r]^{\phi}&\F_p[G], }
\end{align*}
and we label the generators of $\F_p[G]$ by
$\ol{x}_i=\phi(\sigma_i)-1$, for $1\leq i\leq d$.  Letting $f_i$ be
the image of $(\rho_i-1)$ under the identification
$\F_p\{\{x_1,\ldots,x_d\}\}\iso \F_p[F]$, we have
$\ker(\phi)=(f_1,\ldots,f_r)$, and so $\F_p[G]\cong
\F_p\{\{x_1,\ldots,x_d\}\}/(f_1,\ldots,f_r)$.  Let
$I=(\ox_1,\ldots,\ox_d)$ denote the augmentation ideal of $\F_p[G]$,
and define the \emph{level} of an element $f\in \F_p[G]$ to be the
maximal $n$ such that $f\in I^n$.  Let $r_i$ be the number of defining
relations of level $i$, so that in particular we have $\sum r_i=r$.

The proof of the Golod-Shafarevich theorem centers around the exact
sequence of $\F_p[G]$-modules
\begin{align*}
\xymatrix@1{\bigoplus\limits_{i=1}^r\F_p[G]\ar[r]^{J}&\bigoplus\limits_{i=1}^d \F_p[G]\ar[r]^(.55){\psi}&\F_p[G]\ar[r]^(.55)\vare&\F_p\ar[r]&0},
\end{align*}
where we define the three maps as follows:
\begin{itemize}
\item $\vare$ is the augmentation map, which translates under $\phi$
to the ``evaluation at $(x_1,\ldots,x_d)=(0,\ldots,0)$'' map on power
series.
\item $\psi$ is the linear map defined by
\begin{align*}
\psi(g_1,\ldots,g_d)=\sum_{i=1}^d g_i\ox_i,
\end{align*}\
\item To define $J$ we introduce the \emph{Fox partial derivative
operators} $\frac{\partial f}{\partial x_j}$ for $f\in I\subset
\F_p[[F]]$ by observing that if $f\in I$, then $f$ has no constant
term and hence, after collecting the monomials appearing in $f$
according to their last factor, can be written uniquely in the form
$f=\sum \frac{\partial f}{\partial x_j}x_j$.  Now define $J$ (a
``non-commutative Jacobian'') by
\begin{align*}
J(g_1,\ldots,g_r):=\left(\sum_{i=1}^r g_i\phi\left(\frac{\partial f_i}{\partial x_1}\right),\ldots,\sum_{i=1}^r g_i\phi\left(\frac{\partial f_i}{\partial x_d}\right)\right).
\end{align*}
\end{itemize}
\noindent 

The sequence remains exact after taking quotients by suitable powers
of the augmentation ideal, and we arrive at the exact sequences
\begin{align*}
\xymatrix@C=15pt{
0\ar[r]&\ker(J_n)\ar[r]&\bigoplus\limits_{i=1}^r \F_p[G]/I^{n-v(f_i)}\ar[r]^{J_n}&\bigoplus\limits_{i=1}^d \F_p[G]/I^{n-1}\ar[r]^(.55){\psi_n}& \F_p[G]/I^n\ar[r]^(.6){\ol{\vare}}&\F_p\ar[r]& 0.\\
}
\end{align*}
for each $n\geq 1$.  Define
\begin{align*}
c_n:=\dim_{\F_p}\F_p[G]/I^n,\quad\quad e_n:=\dim_{\F_p}\ker J_n
\end{align*}
and set $I^n=\F_p[G]$ for $n\leq 0$, so that $c_n=e_n=0$ for $n\leq
0$.  Finally, recall that $r_i$ was the number of relations of level
$i$ for $i\geq 1$, and we set $r_0=1$ by convention.  Taking the
alternating sum of dimensions of the exact sequence above, and noting
that
\begin{align*}
\dim_{\F_p} \left(\bigoplus_{i=1}^r
\F_p[G]/I_{n-v(f_i)}\right)&=\sum_{i=1}^r c_{n-v(f_i)}=\sum_{i=1}^n
r_ic_{n-i},
\end{align*}
gives the following key result:
\begin{Theo}[Golod-Shafarevich Recursion Relation]\label{recrel}
For a $d$-generated pro-$p$-group $G$, and with all other notation as in the above paragraph, we have
\begin{align*}
\sum_{i=0}^n r_ic_{n-i}-dc_{n-1}=1+e_n
\end{align*}
for all $n\geq 1$.
\end{Theo}

Before stating the Golod-Shafarevich equality, we recall the following
theorem of Jennings which relates the dimension factors
$a_n=\dim_{\F_p}G_n/G_{n+1}$ to the invariants
$c_n=\dim_{\F_p}\F_p[G]/I^n$ defined above.
\begin{Theo}[Jennings, \cite{Je41}]\label{Jenn}
Let $G$ be a finitely-generated pro-$p$-group, and define 
\begin{align*}
b_n:=c_{n+1}-c_n=\dim_{\F_p}I^n/I^{n+1}, \quad\quad
P_n(t):=\frac{1-t^{n}}{1-t^{np}}.
\end{align*}
Then
\begin{align*}
\prod_{n=1}^\infty P_n(t)^{-a_n}=\sum_{n=1}^\infty b_nt^n.
\end{align*}
\end{Theo}

\noindent Collecting all of the above provides us with our desired result.

\begin{Theo}[A Golod-Shafarevich Equality]\label{GSEQ}
Let $G$ be a $d$-generated analytic pro-$p$-group, and take all other
notation as above.  Then
\begin{align*}
\sum_{k=2}^\infty r_kt^k-dt+1&=\prod_{n=1}^\infty P_n(t)^{a_n}+\,\frac{\sum_{n=1}^\infty e_nt^n}{\sum_{n=1}^\infty c_nt^n}
\end{align*}
for all $0\leq t<1$.
\end{Theo}
\begin{proof}
Since $G$ is analytic, the power series $\sum c_nt^n$ converges
absolutely on the unit interval (\cite{La64}), and since $e_n\leq rc_n$ by
definition of the vector space whose dimension it measures, so does
$\sum e_nt^n$.  Absolute convergence now allows us to re-write
\begin{align*}
\left(\sum_{k=0}^\infty r_kt^k-dt\right)\left(\sum_{n=1}^\infty c_nt^n\right)=\sum_{n=1}^\infty\sum_{i=0}^n (r_ic_{n-i}-dc_{n-1})t^n=\sum_{n=1}^\infty (1+e_n)t^n,
\end{align*}
the second equality following from Theorem \ref{recrel}.  Further, we
have $r_0=1$ and $r_1=0$, and so we can re-write $\sum_{k=0}^\infty
r_kt^k-dt=\sum_{k=2}^\infty r_kt^k-dt+1$.  Solving the previous equation for this quantity gives
\begin{align*}
\sum_{k=2}^\infty r_kt^k-dt+1=\frac{\sum_{n=1}^\infty (1+e_n)t^n}{\sum_{n=1}^\infty c_nt^n}=\frac{t}{(1-t)}\cdot\frac{1}{\sum_{n=1}^\infty c_nt^n}+\,\frac{\sum_{n=1}^\infty e_nt^n}{\sum_{n=1}^\infty c_nt^n},
\end{align*}
and the result now follows from
\begin{align*}
\frac{t}{1-t}\cdot \frac{1}{\sum_{n=1}^\infty c_nt^n}=\frac{1}{\sum_{n=0}^\infty b_nt^n}=\prod_{n=1}^\infty P_n(t)^{a_n},
\end{align*}
the last step being Theorem \ref{Jenn}.
\end{proof}
\begin{Remark}\label{Kochform}
Koch's proof in \cite{Ha78} gives $\sum r_kt^k-dt+1>0$, which is
obtained from the current theorem by noting that $e_n\geq 0$ for all
$n$, and that $P_n(t)>0$ for all $t\in(0,1)$.  Either of these
versions allows you to conclude Theorem \ref{medgs}.  One simply
observes that the assumption that $R\subset F_m$ implies $r_k=0$ for
all $k<m$, and that the right-hand side of the equation in Theorem
\ref{GSEQ} is strictly positive, giving
\begin{align*}
\sum_{k=2}^\infty r_kt^k-dt+1=\sum_{k=m}^\infty r_kt^k-dt+1>r_mt^m-dt+1>0.
\end{align*}
This last inequality is violated at
$t=\left(\frac{d}{mr}\right)^{1/(m-1)}$ unless $r\leq
d^m\frac{(m-1)^{m-1}}{m^m}$.
\end{Remark}

For a finite $p$-group $G$, we have $I^n=0$ for sufficiently large $n$
(\cite{Ko02}, Lemma 7.9), and so
\begin{align*}
c_n=\dim_{\F_p}\F_p[G]/I^n=\dim_{\F_p}\F_p[G]=|G|
\end{align*}
for all sufficiently large $n$.  More concretely, Jennings' theorem
implies that $b_n$ ($=c_{n+1}-c_n$) is zero for all $n>N:=(p-1)\sum
na_n$ (the degree of the polynomial $\prod P_n(t)^{-a_n}$), implying
that the sequence $\{c_n\}$ stabilizes after this term.  The
Golod-Shafarevich recursion relation
\begin{align*}
\sum_{i=0}^n r_ic_{n-i}-dc_{n-1}=1+e_n
\end{align*}
in turn implies that $1+e_n=(r+1-d)|G|$ for all sufficiently large
$n$, and hence that the term $\frac{\sum e_nt^n}{\sum c_nt^n}$
appearing on the right-hand side of the Golod-Shafarevich equality is
non-zero for any group with $r\geq d$ (e.g., finite groups).  In the
author's Ph.D. thesis (\cite{Mc08}), this observation is used to give
an improvement on the version of the Golod-Shafarevich inequality
described in Remark \ref{Kochform}.

\begin{Cor}\label{strict}
Let $G$ be a finite $p$-group, let $N=(p-1)\sum_{n=1}^\infty na_n(G)$,
let $m$ be the level of the deepest relation defining $G$, and take all other notation as in Theorem \ref{GSEQ}.  Then
\begin{align*}
\sum_{k=2}^\infty r_kt^k-dt+1>\prod_{n=1}^\infty P_n(t)^{a_n}+(1-d+r)\left(1-\frac{1}{|G|}\right)t^{N+m}>0,
\end{align*}
for all $0\leq t<1$.
\end{Cor}

\section{Quadratic Imaginary Number Fields}\label{quadim}

As discussed in the introduction, the problem of determining the
finiteness of the $p$-tower group $G_K^\infty$ is largely solved in
the case that $K$ is a quadratic imaginary number field and $p$ is an
odd prime.  Namely, the problem is almost completely decided by
$d(G)$, which is computable as the $p$-rank of the class group of $K$:
If $d\leq 1$, $G_K^\infty$ is finite, and if $d\geq 3$, $G_K^\infty$
is infinite.  We are thus left with the case of $d=r=2$, and the
Golod-Shafarevich equality (or Koch's form of the inequality given in
Remark \ref{Kochform}) yields further information in this case.
Namely, since $e_n\geq 0$ and $P_n(t)\geq 0$ for all $n$ and
$t\in(0,1)$, Theorem \ref{GSEQ} gives for all $t\in(0,1)$ the
inequality
\begin{align*}
t^{m_1}+t^{m_2}-2t+1>0,
\end{align*}
where $m_1\leq m_2$ are the levels of the two relations in a minimal
presentation of $G$.  Further, we have that $m_1$ and $m_2$ are both
odd (again by \cite{KV74}) and greater than one (since $r_1=0$).  It
is now easily checked that there are only three choices for the pair
$(m_1,m_2)$ for which the inequality is not violated, i.e., three
possible relation structures for $G_K^\infty$ under the assumption
that the group is finite.  For ease of reference, we will call a
pro-$p$-group \emph{interesting} if it is finite, 2-generated and
2-related, and has relations only in odd levels.
\begin{Theo}[Koch-Venkov]\label{KV}
If $G$ is an interesting pro-$p$-group with relations in levels $m_1$
and $m_2$ with $m_1\leq m_2$, then we have
\begin{align*}
(m_1,m_2)\in\{(3,3),(3,5),(3,7)\}.
\end{align*}
We call this pair the \emph{Zassenhaus type} (or just \emph{Z-type}) of the group.
\end{Theo}
This classification of Z-types for interesting $p$-tower groups
provides hope for computationally showing a given $p$-tower group to
be infinite by showing that any relations defining it lie deeper than
the third level.  For example, the author (\cite{Mc09}) used work of
Vogel (\cite{Vo04}) to show that the vanishing of certain traces of
Massey products on the $\F_p$-cohomology of $G_K^\infty$ implies the
infinitude of the group.

Before moving on, we pause to remark on the current state of knowledge
on abstract pro-$p$-groups of these three Z-types (as always, with $p$
odd):
\begin{itemize}
\item \textbf{Z-Type (3,3):} All known finite non-cyclic $p$-tower
groups, dating back to the earliest examples from Scholz and Taussky
(\cite{ST34}), are of this Z-type.  Recently, Bartholdi and Bush
\cite{BB06} constructed and analyzed an infinite series of 3-groups of
$Z$-type $(3,3)$ whose derived lengths tend to infinity, providing the
first explicit candidates for interesting $p$-tower groups of length greater than two.
\item \textbf{Z-Type (3,5):} Skopin (\cite{Sk73}) has found a family
of finite examples of Z-type $(3,5)$, and Koch and Venkov
(\cite{KV74}) were able to find some infinite examples (using a
variant of the Golod-Shafarevich inequality).  No groups in either of
these families have been shown to occur as $p$-tower groups.
\item \textbf{Z-Type (3,7):} No finite $p$-groups of this Z-type are
known.  Skopin (\cite{Sk73}) has placed a lower bound on the size of a
small family of such groups.  We will expand the scope of this result
in Theorem \ref{cardbound} by showing that \emph{any} pro-$p$-group of
this type must be particularly large.
\end{itemize}

Returning to Corollary \ref{strict} (and recalling the notation
therein), we remark that since interesting $p$-tower groups are of one
of only those three possible Z-types, we can take $m=7$ for any
interesting $p$-tower group.  Further, since $N$ depends only on $p$
and the series $\{a_n(G)\}$, Corollary \ref{strict} gives a strict
strengthening of the Golod-Shafarevich inequality without referring to
any new invariants of the group beyond the dimension factors (one can
replace the constant in front of $t^{N+m}$ with $\frac{p-1}{p}$ so
that no knowledge of the order of the group is required).  Motivated
by this observation, we will return to the implications of the
Golod-Shafarevich equality on interesting groups of Z-type $(3,7)$
after extracting more detailed information on the values of the
dimension factors of 2-generated 2-related pro-$p$-groups.

\subsection{Bounds on Dimension Factors}

Of principal importance in determining dimension factors is the
following theorem of Lazard giving an explicit description of the
modular dimension subgroups $G_n$ in terms of the lower central series
(defined recursively by $\gamma_1(G)=G$,
$\gamma_n(G)=[G,\gamma_{n-1}(G)]$).

\begin{Theo}[Lazard, \cite{La54}]\label{lazard}
For any group $G$ and any prime $p$, we have that the $n$-th dimension
subgroup $G_n$ of $G$ is given by
\begin{align*}
G_n=\prod_{ip^j\geq n}\gamma_i(G)^{p^j}.
\end{align*}
\end{Theo}

As a simple consequence, since a surjection of groups $H\ra K$ induces
a surjection $\gamma_n(H)\ra \gamma_n(K)$ for all $n$, we obtain the
following as an immediate corollary of Lazard's theorem.

\begin{Cor}\label{GmodH}
A surjection of groups $H\ra K$ induces surjections
\begin{align*}
\xymatrix{H_n/H_{n+1}\ar@{->>}[r]& K_n/K_{n+1}}
\end{align*}
for all $n\geq 1$.  In particular, surjections $H\ra G$ and $G\ra K$
give the inequalities
\begin{align*}
a_n(H)\geq a_n(G)\geq a_n(K).
\end{align*}
\end{Cor}

We will apply the corollary to bound various dimension factors of an
interesting $p$-tower group $G$ from above and below.  The lower bound
is easiest:

\begin{Prop}\label{lowbound}
Let $G$ be an interesting $p$-tower group with abelianization of type
$(p^a,p^b)$ with $1\leq a\leq b$.  Then
\begin{align*}
a_n(G)\geq \begin{cases}
2&n=p^c,\quad0\leq c\leq a-1\\
1&n=p^c,\quad a\leq c\leq b-1\\
	   \end{cases}
\end{align*}
\end{Prop}
\begin{proof}
We apply Corollary \ref{GmodH} to the surjection $G\ra
G^{ab}\iso\Z/p^a\Z\oplus \Z/p^b\Z$.  For any abelian group $H$, we
have $\gamma_i(H)=1$ for $i\geq 2$, and so the Lazard product formula
for $H_n$ reduces to
\begin{align*}
H_n=\prod_{p^j\geq n}\gamma_1(H)^{p^j}=H^{p^{\lfloor \log_pn\rfloor}}
\end{align*}
In particular $H_n=H_{n+1}$ unless $n$ is a power of $p$, so only
dimension factors with $p$-power indices $p^c$ can be non-trivial.
Applying this to $G^{\text{ab}}\iso \Z/p^a\Z\oplus \Z/p^b\Z$ (written
additively), we have
\begin{align*}
a_{p^c}=\dim_{\F_p}G^{\text{ab}}_{p^c}/G^{\text{ab}}_{p^c+1}=\dim_{\F_p}\left[\frac{p^c\Z/p^a\Z}{p^{c+1}\Z/p^a\Z}\oplus \frac{p^c\Z/p^b\Z}{p^{c+1}\Z/p^b\Z}\right].
\end{align*}
The first factor is non-trivial only for $0\leq c\leq a-1$ and the
second factor is non-trivial only for $0\leq c\leq b-1$, giving the
result.
\end{proof}

For an upper bound, we relate the dimension factors of a group $G$ to
the (more easily calculable) mod $p$ quotients of its lower central
factors.  Define
\begin{align*}
g_n(G):=\dim_{\F_p}\frac{\gamma_n(G)}{\gamma_n(G)^p\gamma_{n+1}}.
\end{align*}
The relation to the dimension factors is then given by
\begin{Lemma}\label{surj}
For any finitely-generated pro-$p$-group $G$, we have $a_n(G)\leq
g_n(G)$ for all $n<p-1$.
\end{Lemma}
\begin{proof}
Write $\gamma_i$ for $\gamma_i(G)$.  It suffices to demonstrate a
surjection of $\F_p$ vector spaces $\gamma_n/\gamma_n^p\gamma_{n+1}\ra
G_n/G_{n+1}$.  The assumption that $p>n$ renders most of the terms in
Lazard's product formula for $G_n$ redundant.  Namely, noting the
inclusions $\gamma_i\leq \gamma_j$ for $i\geq j$ and
$\gamma_i^{p^j}\leq \gamma_i^{p^k}$ for $j\geq k$, we claim that the
product simplifies to
\begin{align*}
G_n=\prod_{ip^j\geq n}\gamma_i^{p^j}=G^p \gamma_n.
\end{align*}
To see this, observe that any factor in the product must either have
$i\geq n$, in which case that factor is contained in $\gamma_n$, or
$j\geq 1$, in which case the factor is contained in $\gamma_1^p=G^p$.
Similarly, since $p>n+1$, repeating the argument gives
$G_{n+1}=G^p\gamma_{n+1}$.  Now the kernel of the natural quotient map
\begin{align*}
\xymatrix@1{
\gamma_n\ar@{->>}[r]& \frac{\gamma_n}{(G^p\cap\gamma_n)\gamma_{n+1}}\cong\frac{G^p\gamma_n}{G^p\gamma_{n+1}}=\frac{G_n}{G_{n+1}}
}\end{align*}
clearly contains $\gamma_n^p\gamma_{n+1}$, giving the desired surjection.
\end{proof}

This result in hand, we now recall that any interesting $p$-tower
group $G$ admits a presentation $F/\<\rho_1,\rho_2\>$ where $\rho_1$
is of level 3 with respect to the Zassenhaus filtration and $\rho_2$
is of level $i$ for $i\in\{3,5,7\}$.  Regardless of the level of
$\rho_2$, $G$ is thus a quotient of the one-relator pro-$p$-group
$\widetilde{G}:=F/\<\rho_1\>$ whose single relation lies in level 3.
Such groups were studied extensively by Labute, especially in regard
to their lower central series.  Important to our upper bound will be
his calculation of the lower central factors of a $d$-generated
one-relator pro-$p$-group with one relation in level $k$
(\cite{La70}):
\begin{align*}
g_n(\widetilde{G})=\frac{1}{n}\sum_{j\mid n}\mu\left(\frac{n}{j}\right)\left[\sum_{0\leq i\leq \lfloor\frac{j}{k}\rfloor}(-1)^i\frac{j}{j+(1-k)i}\binom{j+(1-k)i}{i}d^{j-ki}\right],
\end{align*}
where $\mu$ denotes the Moebius function.  For an interesting
$p$-tower group $G$, combining Corollary \ref{GmodH} (applied to the
natural surjection $\widetilde{G}\ra G$) and Lemma \ref{surj} gives
the chain of inequalities $a_n(G)\leq a_n(\widetilde{G})\leq
g_n(\widetilde{G})$, proving the following proposition.

\begin{Prop}\label{an}
Let $G$ be an interesting $p$-tower group.  Then for $p>n+1$ we have
\begin{align*}
a_n(G)\leq \frac{1}{n}\sum_{j\mid n}\mu\left(\frac{n}{j}\right)\left[\sum_{0\leq i\leq \lfloor\frac{j}{3}\rfloor}(-1)^i\frac{j}{j-2i}\binom{j-2i}{i}2^{j-3i}\right]
\end{align*}
For $p>7$, this gives the following table of upper bounds for the
first few dimension factors of an interesting $p$-tower groups:
\begin{align*}
\begin{tabular}{|c|c|c|c|c|c|c|c|c|c|}\hline
$n=$&1&2&3&4&5&6&7&8&9\\\hline
$a_n\leq$&2&1&1&1&2&2&4&5&8\\\hline
\end{tabular}
\end{align*}
\end{Prop}
\noindent We can refine this slightly for groups of $Z$-type $(3,7)$.
\begin{Lemma}\label{a7}
For $p>7$ and an interesting $p$-tower group $G$ of $Z$-type $(3,7)$,
we have $a_7(G)\leq 3$.
\end{Lemma}
\begin{proof}
Consider a minimal presentation $1\ra R\ra F\ra G\ra 1$, and choose
generators $\rho_1,\rho_2$ for $R$ of respective levels 3 and 7.  Let
$\widetilde{G}=F/\<\rho_1\>$.  We have the commutative diagram
\begin{align*}
\xymatrix{
0\ar[r]&\<\rho_1\>/(\<\rho_1\>\cap F_8)\ar@{_(->}[d]\ar[r]&F_7/F_8\ar@{=}[d]\ar[r]&\widetilde{G}_7/\widetilde{G}_8\ar@{->>}[d]\ar[r]&0\\
0\ar[r]&R/(R\cap F_8)\ar[r]&F_7/F_8\ar[r]&G_7/G_8\ar[r]&0
}
\end{align*}
of finite-dimensional $\F_p$-vector spaces.  By assumption that
$\rho_1$ and $\rho_2$ form a minimal system of generators for
$RF_8/F_8\iso R/(R\cap F_8)$ (in particular, since $\rho_2\notin
\<\rho_1\>\cap F_8$), we have $\dim_{\F_p}R/(R\cap
F_8)=\dim_{\F_p}\<\rho_1\>/(\<\rho_1\>\cap F_8)+1$.  This then gives
$a_7(G)\leq a_7(\widetilde{G})-1= 3$ by the previous proposition.
\end{proof}

Finally, we return to the implications of the Golod-Shafarevich
equality for interesting $p$-tower groups of Z-type $(3,7)$.  A key
motivation is that the polynomial $t^7+t^3-2t+1$ has a minimum value
of about $0.02$ on the unit interval, implying that the
Golod-Shafarevich inequality \emph{nearly} prohibits analytic groups
of this $Z$-type from occurring.  While the Golod-Shafarevich results
do not rule out the existence of such groups, we instead obtain a
rather large lower bound on their orders.

\begin{Theo}\label{cardbound}
Let $p>7$, and suppose $G$ is a pro-$p$-group of $Z$-type $(3,7)$
whose abelianization is of type $(p^a,p^b)$ with $1\leq a\leq b$.
Then $|G|\geq p^{21+a+b}\geq p^{23}$.  
\end{Theo}
\begin{proof}
Using that $e_n\geq 0$ for all $n$, the Golod-Shafarevich equality implies that
\begin{align*}
t^7+t^3-2t+1>\prod_{n=1}^\infty P_n(t)^{a_n}
\end{align*}
for all $t\in(0,1)$.  We abbreviate $P_n(t)$ by $P_n$, and begin by
breaking up the right-hand product (making use of Proposition
\ref{lowbound} and that $a_1=d=2$ in the process):
\begin{align*}
\prod_{n=1}^\infty P_n^{a_n}=P_1^2P_2^{a_2}\cdots
P_9^{a_9}\cdot\prod_{c=1}^{a-1}P_{p^c}^2\cdot\prod_{c=a}^{a+b-1}P_{p^c}\cdot\prod'P_n^{a_n}
\end{align*}
where the primed product at the end consists of all terms not
explicitly pulled out in one of the other displayed factors.  (The
reason for specifically pulling out the first nine terms will become
clear by the end of the proof).  The products with $p$-power indices
now telescope to give
\begin{align*}
\prod_{n=1}^\infty
P_n^{a_n}&=\left(\frac{1-t}{1-t^p}\right)^2P_2^{a_2}\cdots
P_9^{a_9}\cdot\left(\frac{(1-t^p)^2}{(1-t^{p^a})(1-t^{p^b})}\right)\cdot\prod'P_n^{a_n}\\
&\geq (1-t)^2(1-t^2)^{a_2}\cdots(1-t^9)^{a_9}\prod'(1-t^n)^{a_n}.
\end{align*}
Since $|G|=p^{\sum a_n}$, we now search for the sequence
$a_2,a_3,\ldots$ of dimension factors which gives the smallest value
of $A:=\sum a_n$ subject to this last inequality and whose terms
satisfy the constraints of Proposition \ref{an} and Lemma \ref{a7}.
Since $(1-t^m)\geq (1-t^n)$ on the unit interval whenever $m\geq n$,
we note that of all the sequences which sum to $A$, the one with the
smallest value of $\prod (1-t^n)^{a_n}$ is the one with $a_1=A$,
$a_i=0$ for $i\geq 2$.  Elementary calculus or a graphing calculator
shows that for $A=2$, the inequality $t^7+t^3-2t+1>(1-t)^2$ is
violated at $t=.5$.  Thus we must have $A>2$, but now the restriction
that $a_1\leq 2$ implies that the minimal possible solution occurs
when $a_1=2$, $a_2=A-2$.  A similar argument rules out $A=3$, and
since $a_2\leq 1$, the minimal solution must have $a_3$ also
non-trivial.  We now repeat, incrementing the next smallest dimension
factor until either the Golod-Shafarevich equality is satisfied, or we
reach the upper bound on that factor prescribed by Proposition
\ref{an} or Lemma \ref{a7}.  The process terminates after incrementing
$a_9$ to 6, which one can verify by calculating that
\begin{align*}
t^7+t^3-2t+1>(1-t)^2(1-t^2)(1-t^3)(1-t^4)(1-t^5)^2(1-t^6)^2(1-t^7)^3(1-t^8)^5(1-t^9)^5
\end{align*}
is violated for $t=.55$, whereas the analogous inequality with $a_9=6$
holds for all $t\in(0,1)$.  Recalling that this sequence of dimension
factors (including those of $p$-power index dealt with earlier in the
proof) was the sequence which gave the \emph{minimal} possible value
of $A$, we conclude that for any interesting $p$-tower group of
$Z$-type $(3,7)$, we have
\begin{align*}
\sum_{n=1}^\infty a_n\geq 2+1+1+1+2+2+3+5+6+2(a-1)+(b-a)=21+a+b.
\end{align*}
\end{proof}
\begin{Remark}
The explicit bound $|G|\geq p^{23}$ might suggest that to find
candidates for quadratic imaginary number fields whose $p$-tower group
is of $Z$-type (3,7), it would be prudent to search for fields with
very large $p$-class groups (but still, of course, with $p$-rank 2).
The bound $|G|\geq p^{21+a+b}$, however, shows that this is not the
case.  In particular, since we have assumed that $|G^{ab}|=p^{a+b}$,
the theorem implies (via the inequality $|G'|\geq p^{21}$) that it is
the commutator subgroups which contains the bulk of this newfound
size.
\end{Remark}

\section{Concluding Remarks}

The proof of Theorem \ref{cardbound} suggests that we can hope to
better understand finite groups of $Z$-type (3,7) by considering
abstract sequences of invariants which conform to the bounds given by
the various results on such a group's dimension factors.  Namely, for
a sequence $a_1,a_2,\ldots$ of non-negative integers (and a fixed
prime $p$), we could define invariants $b_n$, $c_n$, and $e_n$ for
$n\geq 1$ by mirroring their definitions as found in the text:
\begin{align*}
\sum_{n=0}^\infty b_nt^n=\prod_{n=1}^\infty P_n(t)^{a_n},\quad\quad c_n=b_{n+1}-b_n,\quad\end{align*}
\begin{align*}
\quad e_n=c_n-2c_{n-1}+c_{n-3}+c_{n-7}-1.
\end{align*}
We call the original sequence $a_n$ \emph{valid} if its terms satisfy
the bounds given by Proposition \ref{an} and Lemma \ref{a7}, and if
the corresponding sequences $c_n$ and $e_n$ are non-negative and
stabilize for sufficiently large $n$.  Certainly a necessary condition
for the existence a finite $p$-group of Z-type (3,7) is the existence
of such a valid sequence.  While it may be tempting to view Theorem
\ref{cardbound} as a first step toward proving the non-existence of
such a group, the following example, found in collaboration with Ray
Puzio, shows that this interpretation is premature.

\begin{Exam}
For $p=17$, the sequence
\begin{align*}
\{a_n\}_{n=1}^{15}=\{2,1,1,1,2,2,3,3,4,4,6,5,7,5,4\}.
\end{align*}
is valid.
\end{Exam}
It seems difficult, at the present, to determine whether or not this
series occurs as the sequence of dimension factors of an interesting
$p$-tower group, though we note that by Proposition \ref{GmodH} the
abelianization of such a group would necessarily be isomorphic to
$\Z/17\Z\oplus\Z/17\Z$.  Further, by summing the sequence, we see that
such a group would have order $17^{50}$, well beyond the bound
guaranteed by Theorem \ref{cardbound} (and larger than the monster
group!).

Finally, we wish to remark on a possible alternate interpretation for
the sequence of invariants $e_n$ appearing in the Golod-Shafarevich
equality.  A result of Labute (\cite{La06}, Theorem 5.1g) shows that
for a mild pro-$p$-group, one has
\begin{align*}
\sum r_kt^k-dt+1=\sum b_nt^n,
\end{align*}
which after applying Jennings' theorem to the right-hand side, is
precisely the Golod-Shafarevich inequality in the case that $e_n=0$
for all $n$.  This suggests a further interpretation of the $e_n$ as a
measure of the non-mildness of a pro-$p$-group $G$.  As a toy example
of this interpretation, the fact that for a finite $p$-group $G$ we
have $e_n=|G|-1\neq 0$ for all sufficiently large $n$ (see the
discussion after Remark \ref{Kochform}) might suggest that finite
$p$-groups are ``highly non-mild.''

\section*{Acknowledgments}

The author would like to thank William McCallum for his guidance and
support, Ray Puzio for his insights on recursion relations, and Kirti
Joshi, Dinesh Thakur, and Klaus Lux for endless valuable conversations.

\end{document}